\documentclass[10pt]{article}

\usepackage{lmodern}
\usepackage[T1]{fontenc}
\usepackage{amsmath}
\usepackage{amsthm}
\usepackage{amssymb}
\usepackage{mathabx} 
\usepackage{url}
\usepackage{latexsym}
\usepackage{titlefoot}
\usepackage[small]{titlesec}
\usepackage{units} 
\usepackage[small,it]{caption}

\usepackage{setspace}

\usepackage[square,comma,numbers,sort&compress]{natbib}

\usepackage{tikz}
\usetikzlibrary{patterns, intersections}
\usepackage{xparse} 

\usepackage{footmisc} 

\usepackage{color}
\definecolor{lightgray}{rgb}{0.8, 0.8, 0.8}
\definecolor{darkgray}{rgb}{0.7, 0.7, 0.7}

\usepackage[bookmarks]{hyperref}
\hypersetup{
        colorlinks=true,
        linkcolor=black,
        anchorcolor=black,
        citecolor=black,
        urlcolor=black,
        pdfpagemode=UseThumbs,
        pdftitle={On The Dimension of Downsets of Integer Partitions and Compositions},
        pdfsubject={Combinatorics},
        pdfauthor={Engen and Vatter},
}

\newcounter{todocounter}

\theoremstyle{plain}
\newtheorem{theorem}{Theorem}[section]
\newtheorem{proposition}[theorem]{Proposition}

\newtheorem{corollary}[theorem]{Corollary}

\theoremstyle{definition}

\makeatletter
\newfont{\footsc}{cmcsc10 at 8truept}
\newfont{\footbf}{cmbx10 at 8truept}
\newfont{\footrm}{cmr10 at 10truept}
\renewenvironment{abstract}%
                {
                  \begin{list}{}%
                     {\setlength{\rightmargin}{1in}%
                      \setlength{\leftmargin}{1in}}%
                   \item[]\ignorespaces\begin{small}}%
                 {\end{small}\unskip\end{list}}

\newcommand{\st}{\::\:}

\newcommand{\Age}{\operatorname{Age}}
\newcommand{\C}{\mathcal{C}}
\newcommand{\D}{\mathcal{D}}

\newcommand{\mbar}{\widebar{m}}
\newcommand{\mbarbar}{\bar{\mbar}}

\newcommand\absdot[2]{
	\node at #1 {\normalsize $\bullet$};
	\node at #1 [below] {$#2$};
}

\newcommand{\hassestp}[2]{
  \foreach \c [count=\i] in {#1} {
    \node (\i0) at (\i,0) {$\c$};
  };
  \foreach \c [count=\i] in {#2} {
    \node (\i1) at (\i,1) {$\c$};
    \foreach \d [count=\j] in {#1} {
      \ifthenelse{\i < \j \OR \i > \j}{\draw (\j0)--(\i1);}{}
    };
  };
}

\datefoot{\today}
\amssubj{06A07, 68R15}

\newpagestyle{main}[\small]{
        \headrule
        \sethead[\usepage][][]
        {\sc On The Dimension of Downsets of Integer Partitions and Compositions}{}{\usepage}}

\title{\sc On The Dimension of Downsets of Integer Partitions and Compositions}
\author{%
Michael Engen and Vincent Vatter%
\footnote{Vatter's research was sponsored by the National Security Agency under Grant Number H98230-16-1-0324.  The United States Government is authorized to reproduce and distribute reprints not-withstanding any copyright notation herein.}\\[-0.25ex]
\small Department of Mathematics\\[-0.5ex]
\small University of Florida\\[-0.5ex]
\small Gainesville, Florida USA\\[-1.5ex]
}

\titleformat{\section}
        {\large\sc}
        {\thesection.}{1em}{}

\date{}

\newcommand\mybullet{\raisebox{-5pt}{\normalsize \ensuremath{\bullet}}}
\newcommand\mycirc{\raisebox{-5pt}{\normalsize \ensuremath{\circ}}}

\makeatletter
\def\absdot{\@ifnextchar[{\@absdotlabel}{\@absdotnolabel}}
	\def\@absdotlabel[#1]#2{%
		\node at #2 {\normalsize \mybullet};
		\node at #2 [below=2pt] {\ensuremath{#1}};
	}
	\def\@absdotnolabel#1{%
		\node at #1 {\normalsize \mybullet};
	}
\def\absdothollow{\@ifnextchar[{\@absdothollowlabel}{\@absdothollownolabel}}
	\def\@absdothollowlabel[#1]#2{%
		\node at #2 {\normalsize \textcolor{white}{\mybullet}};
		\node at #2 {\normalsize \mycirc};
		\node at #2 [below=2pt] {\ensuremath{#1}};
	}
	\def\@absdothollownolabel#1{%
		\node at #1 {\normalsize \textcolor{white}{\mybullet}};
		\node at #1 {\normalsize \mycirc};
	}
\makeatother

\newcommand{\plotperm}[1]{
	\foreach \j [count=\i] in {#1} {
		\absdot{(\i,\j)};
	};
}

\newcommand{\plotpermgraph}[1]{
	\foreach \j [count=\i] in {#1} {
		\foreach \b [count=\a] in {#1} {
			\ifthenelse{\a<\i \AND \b>\j}{\draw (\a,\b)--(\i,\j);}{}
		};
	};
	\plotperm{#1};
}

\newcommand{\plotpermdyckpath}[1]{
	\draw[ultra thick, line cap=round] (0.5,0.5)
	\foreach \step in {#1} {
		\ifnum\step=1
			-- ++(0,1)
		\else
			-- ++(1,0)
		\fi
	};
}

\newcommand{\plotdyckpath}[1]{
	\draw[ultra thick, line cap=round] (0.5,0)
	\foreach \step in {#1} {
		\ifnum\step=1
			-- ++(1,1)
		\else
			-- ++(1,-1)
		\fi
	};
}

\newcommand{\arcskinnyplain}[2]{
	\draw[thick] (#1,0) arc (180:0:{(#2-#1)/2});
}

\newcommand{\matchsmall}[1]{
	\begin{tikzpicture}[scale=.1, anchor=base]
		\def\h{0};
		\def\maxh{0};
		\foreach \i/\j in {#1} {
			\pgfmathparse{\j-\i};
			\let\h\pgfmathresult;
			\pgfmathifthenelse{\h>\maxh}{\h}{\maxh};
			\global\let\maxh\pgfmathresult;
		};
		\pgftransformyscale{{4.5/\maxh}};
		\foreach \i/\j in {#1} {
			\arcskinnyplain{\i}{\j};
		};
	\end{tikzpicture}
}

\newcommand{\matchpermsmall}[1]{
	\begin{tikzpicture}[scale=.1, anchor=base]
		\foreach \j [count=\n] in {#1} {};
		\def\h{0};
		\def\maxh{0};
		\foreach \j [count=\i] in {#1} {
			\pgfmathparse{2*\n+1-\j-\i};
			\let\h\pgfmathresult;
			\pgfmathifthenelse{\h>\maxh}{\h}{\maxh};
			\global\let\maxh\pgfmathresult;
		};
		\pgftransformyscale{{4.5/\maxh}};
		\foreach \j [count=\i] in {#1} {
			\arcskinnyplain{\i}{{2*\n+1-\j}};
		};
	\end{tikzpicture}
}

\newcommand{\plotskyline}[1]{
	\foreach \j [count=\i] in {#1} {
		\draw [thick, line cap=round] (\i-1,0)--(\i-1,\j)--(\i,\j)--(\i,0)--cycle;
		\foreach \k in {1,2,...,\j}{
			\draw [thick, line cap=round] (\i-1,\k-1)--(\i,\k-1);
		};
	};
}

\newcommand{\plotskylineshaded}[1]{
	\foreach \j [count=\i] in {#1} {
		\draw [lightgray, fill=lightgray, thick, line cap=round] (\i-1,0)--(\i-1,\j)--(\i,\j)--(\i,0)--cycle;
		\foreach \k in {1,2,...,\j}{
			\draw [lightgray, fill=lightgray, thick, line cap=round] (\i-1,\k-1)--(\i,\k-1);
		};
	};
}

\begin{document}
\maketitle

\pagestyle{main}

\begin{abstract}
We characterize the downsets of integer partitions (ordered by containment of Ferrers diagrams) and compositions (ordered by the generalized subword order) which have finite dimension in the sense of Dushnik and Miller. In the case of partitions, while the set of all partitions has infinite dimension, we show that every proper downset of partitions has finite dimension. For compositions we identify four minimal downsets of infinite dimension and establish that every downset which does not contain one of these four has finite dimension.
\end{abstract}

\section{Introduction}
\label{sec-intro}

The notion of the \emph{dimension} of a poset $P = (X, \le)$ was introduced by Dushnik and Miller~\cite{dushnik:partially-order:}, who defined it as the least $d$ so that $P$ embeds into a product of $d$ linear orders. In particular, the dimension of a countable poset $P$ is the least $d$ so that $P$ embeds into $\mathbb{R}^d$, the definition given by Ore~\cite{ore:theory-of-graph:}. Here we consider the dimension of downsets of integer partitions and compositions.

The partial order on partitions we consider is simply the one in Young's lattice, namely containment of Ferrers diagrams, and we establish the result below.

\begin{theorem}
\label{thm-part-dim}
A downset of integer partitions is finite dimensional if and only if it does not contain every partition.
\end{theorem}

We go on to study the dimension of downsets of compositions under the \emph{generalized subword order}. In this order we view compositions as words over the positive integers $\mathbb{P}$, and we denote the set of these words by $\mathbb{P}^\ast$. Given two compositions $u=u(1)\cdots u(k)$ and $w=w(1)\cdots w(n)$, we say that $u$ is \emph{contained} in $w$ and write $u\le w$ if there are indices $1\le i_1<\cdots<i_k\le n$ such that $u(j)\le w(i_j)$ for all $j$.

This order can be illustrated graphically by way of skyline diagrams. The \emph{skyline diagram}  of the composition $w=w(1)\cdots w(n)$ consists of $n$ columns of cells, with the $i$th column having $w(i)$ cells. For compositions $u$ and $w$, we have $u\le w$ if the skyline diagram of $u$ can be embedded into that of $w$. For example, the diagrams below show that $3413\le 141421143$.

\begin{center}
\begin{tabular}{ccc}
  \begin{tikzpicture}[scale=.35, baseline=(current bounding box.center)]
    \plotskyline{3,4,1,3};
  \end{tikzpicture}
  &
  \begin{tikzpicture}[baseline=(current bounding box.center)]
    \node at (0,0) {$\le$};
  \end{tikzpicture}
  &
  \begin{tikzpicture}[scale=.35, baseline=(current bounding box.center)]
  	\plotskylineshaded{0,3,0,4,1,0,0,3,0};
    \plotskyline{1,4,1,4,2,1,1,4,3};
  \end{tikzpicture}
\end{tabular}
\end{center}

The generalized subword order on compositions has received some attention since it was first considered by Bergeron, Bousquet-M\'elou, and Dulucq~\cite{bergeron:standard-paths-:}, who studied saturated chains in this poset. Snellman~\cite{snellman:standard-paths-:} extended their work. Later, Sagan and Vatter~\cite{sagan:the-mobius-func:} determined the M\"obius function of this poset, and Bj\"orner and Sagan~\cite{bjorner:rationality-of-:comp} showed that this M\"obius function has a rational generating function. Finally, Vatter~\cite{vatter:reconstructing-:} considered the analogue of the Reconstruction Conjecture in this poset.

To state the analogue of Theorem~\ref{thm-part-dim} for compositions, we need to introduce a bit more notation and extend our viewpoint to include compositions. A possibly infinite composition is represented by a word over the alphabet $\mathbb{P}\cup\{n^\omega\st n\in\mathbb{P}\}\cup\{\omega,\omega^\omega\}$. In such a word, $n^\omega$ stands for an infinite number of parts all equal to $n$, $\omega$ stands for an infinite part, and $\omega^\omega$ stands for an infinite number of infinite parts. Given a word $u$ over the alphabet $\mathbb{P}\cup\{n^\omega\st n\in\mathbb{P}\}\cup\{\omega,\omega^\omega\}$, the \emph{age} of $u$, denoted $\Age(u)$ is the set of all compositions which embed into it (this term dates to Fra{\"{\i}}ss\'{e}~\cite{fraisse:sur-lextension-:}). For example, $\Age(\omega\omega\omega)$ is the set of compositions with at most three parts, $\Age(2^\omega)$ is the set of all compositions with all parts at most two, and $\Age(1^\omega \omega 2131^\omega)$ consists of all compositions which embed into the skyline diagram below.

\begin{center}
  \begin{tikzpicture}[scale=.35]
    \plotskyline{1,1,1,4,2,1,3,1,1,1};
		
    \filldraw[black] (-0.75,0.5) circle [radius=0.05cm];
    \filldraw[black] (-0.5, 0.5) circle [radius=0.05cm];
    \filldraw[black] (-0.25,0.5) circle [radius=0.05cm];
    
    \filldraw[black] (3.5,4.25)  circle [radius=0.05cm];
    \filldraw[black] (3.5, 4.5)  circle [radius=0.05cm];
    \filldraw[black] (3.5,4.75)  circle [radius=0.05cm];
    
    \filldraw[black] (10.25,0.5) circle [radius=0.05cm];
    \filldraw[black] (10.5 ,0.5) circle [radius=0.05cm];
    \filldraw[black] (10.75,0.5) circle [radius=0.05cm];
  \end{tikzpicture}
\end{center}

We can now state our result for compositions.

\begin{theorem}
\label{thm-comp-dim}
A downset of compositions in the generalized subword order is finite dimensional if and only if it does not contain $\Age(\omega\omega\omega)$, $\Age(1^\omega 2 1^\omega 2 1^\omega)$, $\Age(\omega 1^\omega \omega 1^\omega)$, or $\Age(1^\omega\omega 1^\omega\omega)$.
\end{theorem}

We also use the concept of ages in the partition setting, where the age of a word $u$ over the alphabet $\mathbb{P}\cup\{n^\omega\st n\in\mathbb{P}\}\cup\{\omega,\omega^\omega\}$ is the set of all (finite) integer partitions which embed into $u$. While notationally identical, it will always be clear from the context whether an age consists of partitions or compositions.

Dimension is a monotone property in that the dimension of a poset is at least that of any of its subposets. Thus to show that a poset is infinite dimensional we show that it contains subposets of arbitrarily large dimension. In particular, we recall that the \emph{crown} on the $2n$ elements $\{a_1,\dots,a_n,b_1,\dots,b_n\}$ is the poset in which the only comparisons are of the form $a_i < b_j$ for $i\ne j$, as depicted in the Hasse diagram below.
\begin{center}
  \begin{footnotesize}
    \begin{tikzpicture}[xscale=2]
    \hassestp{a_1, a_2, \cdots, a_n}{b_1, b_2, \cdots, b_n}
  \end{tikzpicture}
  \end{footnotesize}
\end{center}
It is easily seen that the crown on $2n$ elements has dimension $n$, so we refer to it as the crown of dimension $n$.

To establish that the poset of all integer partitions is infinite dimensional, it suffices to find arbitrarily large crowns of partitions. One such family of crowns is defined by taking
\[
  a_i = (n-i)^i
  \quad\text{and}\quad
  b_i = \bigvee_{j\neq i} a_j,
\]
i.e., taking $a_i$ to be the partition consisting of $i$ parts equal to $n-i$ and $b_i$ to be the join (in Young's lattice) of all $a_j$ for $j\neq i$.

Similarly, one direction of Theorem~\ref{thm-comp-dim} can be established by finding arbitrarily large crowns in the four stated ages. For example, we see that $\Age(\omega\omega\omega)$ contains the crown of dimension $n-3$ shown below for all $n\ge 5$.
  
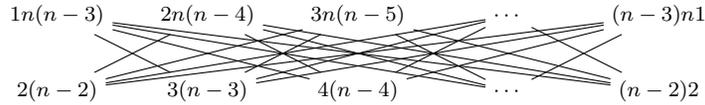
\begin{center}
  \begin{footnotesize}
  \begin{tikzpicture}[xscale=2]
    \hassestp{2 (n-2), 3 (n-3), 4 (n-4), \cdots, (n-2) 2}{1 n (n-3), 2 n (n-4), 3 n (n-5), \cdots, (n-3) n 1}
  \end{tikzpicture}
  \end{footnotesize}
  \end{center}
A slight modification of this crown shows that $\Age(1^\omega 2 1^\omega 2 1^\omega)$ is infinite dimension, as it contains the crown of dimension $n-3$ shown below for all $n\ge 5$.
  
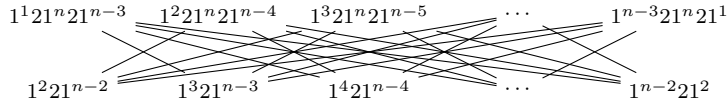
\begin{center}
  \begin{footnotesize}
  \begin{tikzpicture}[xscale=2]
    \hassestp{1^2 2 1^{n-2}, 1^3 2 1^{n-3}, 1^4 2 1^{n-4}, \cdots, 1^{n-2} 2 1^2}{1^1 2 1^n 2 1^{n-3}, 1^2 2 1^n 2 1^{n-4}, 1^3 2 1^n 2 1^{n-5}, \cdots, 1^{n-3} 2 1^n 2 1^1}
  \end{tikzpicture}
  \end{footnotesize}
  \end{center}
The last two ages stated in Theorem~\ref{thm-comp-dim} are isomorphic, so it suffices to show that $\Age(\omega 1^\omega \omega 1^\omega)$ is infinite dimensional. This age contains the crown of dimension $n-1$ shown below for all $n\ge 3$.
  
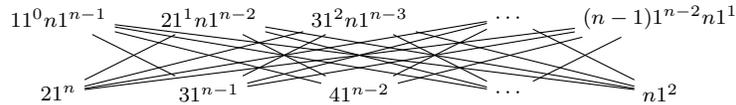
\begin{center}
  \begin{footnotesize}
  \begin{tikzpicture}[xscale=2]
    \hassestp{2 1^n, 3 1^{n-1}, 4 1^{n-2}, \cdots, n 1^2}{1 1^0 n 1^{n-1}, 2 1^1 n 1^{n-2}, 3 1^2 n 1^{n-3}, \cdots, (n-1) 1^{n-2} n 1^1}
  \end{tikzpicture}
  \end{footnotesize}
  \end{center}

Thus it suffices to prove that downsets of compositions not containing any of these four ages are finite dimensional. Note that $\Age(2^\omega)$ is infinite dimensional---this follows from the fact that it contains $\Age(1^\omega 2 1^\omega 2 1^\omega)$, or more easily by observing that it contains the crown of dimension $n$ defined by $a_i=1^{i-1}21^{n-i}$ and $b_i=2^{i-1}12^{n-i}$. Consequently, the age of any (infinite) composition which includes any symbol of the form $\omega^\omega$ or $n^\omega$ for $n\ge 2$ is necessarily infinite dimensional. Therefore when characterizing the finite dimensional ages of infinite compositions we may restrict our attention to ages of words over the alphabet $\mathbb{P}\cup\{1^\omega,\omega\}$.

\section{Tools}
\label{sec-tools}

In this section we introduce the tools we use to establish the other directions of Theorems~\ref{thm-part-dim} and \ref{thm-comp-dim}. A poset $P$ is \emph{well quasi-ordered} if it contains neither infinite antichains nor infinite strictly decreasing chains, i.e., $x_0 > x_1 > \cdots$. We begin by recalling the following well-known result.

\begin{theorem}{(Higman's Lemma~\cite{higman:ordering-by-div:})}
\label{thm-higman}
If $(P,\le)$ is well quasi-ordered then $P^\ast$, the poset of words over $P$ ordered by the generalized subword order, is also well quasi-ordered.
\end{theorem}

As the poset of partitions is a subposet of $\mathbb{P}^\ast$ and the poset of compositions is precisely the poset $\mathbb{P}^\ast$, Higman's Lemma implies that both posets are well quasi-ordered. This allows us to appeal to the following result.

\begin{proposition}
\label{prop-wqo-subclasses-dcc}
Downsets of well quasi-orders satisfy the \emph{descending chain condition}, i.e., there does not exist a sequence of downsets satisfying $\C^0\supsetneq \C^1\supsetneq \C^2\supsetneq \cdots$.
\end{proposition}
\begin{proof}
Suppose to the contrary that the well quasi-ordered downset $\C$ were to contain an infinite strictly decreasing sequence of subdownsets $\C=\C^0\supsetneq \C^1\supsetneq \C^2\supsetneq\cdots$. For each $i\ge 1$, choose $x_i\in\C^{i-1}\setminus\C^i$. The set of minimal elements of $\{x_1,x_2,\ldots\}$ is an antichain and therefore finite, so there is an integer $m$ such that $\{x_1,x_2,\ldots,x_m\}$ contains these minimal elements. In particular, $x_{m+1}\ge x_i$ for some $1\le i\le m$. However, we chose $x_{m+1}\in\C^m\setminus\C^{m+1}$, and because $x_{m+1}\ge x_i$, $x_{m+1}$ does not lie in $\C^i$ and thus cannot lie in $\C^m$, a contradiction.
\end{proof}

Because of Proposition~\ref{prop-wqo-subclasses-dcc}, we can consider a minimal (with respect to set containment) counterexamples to prove Theorems~\ref{thm-part-dim} and \ref{thm-comp-dim}. Our next result shows that such minimal counterexamples cannot be unions of two proper subdownsets, but before proving it we need to make some more general remarks about dimension, and in particular, our approach to establishing that downsets are finite dimensional.

A \emph{realizer} of the poset $P$ is a collection $\mathcal{R}$ of linear extensions of the poset such that $x \le_P y$ if and only if $x \le_L y$ for each $L \in \mathcal{R}$. Given that the elements of a realizer are extensions of the original poset, this is equivalent to saying that for each pair $x,y \in P$ of incomparable elements, there is some $L \in \mathcal{R}$ such that $y \le_L x$.

A \emph{refinement} of the poset $P$ is another partial order, say $\le_R$, such that $x\le_R y$ for all pairs $x,y\in P$ with $x\le_P y$. Because every refinement can be extended to a linear extension, to establish that the dimension of the poset $P$ is at most $n$, it suffices to find a collection $\mathcal{R}$ of $n$ refinements of $P$ such that $x\le_P y$ if and only if $x\le_R y$ for each $R\in\mathcal{R}$. Frequently we go a step further than this. As every refinement of a subposet of $P$ can be extended to a linear extension of $P$, to show that $P$ has dimension at most $n$ it suffices to find a collection $\mathcal{R}$ of $n$ \emph{partial refinements} (meaning refinements of subposets of $P$) with this property.

In constructing and analyzing these refinements or partial refinements, we use two additional terms. If the refinements $R_1$ and $R_2$ satisfy $x<_{R_1} y$ and $y<_{R_2} x$ (or vice versa) then we say that the pair $R_1$, $R_2$ \emph{breaks} the incomparison between $x$ and $y$. Finally, every homomorphism between a poset (or subposet of it) to a totally ordered set (typically $\mathbb{N}$ here) induces a refinement or partial refinement on the poset. In this situation we often say that the induced refinement \emph{sorts} the objects of $P$ according to the homomorphism. For example, a natural refinement of the either the poset of partitions or of compositions is the one that sorts them according to length (number of parts).

\begin{proposition}
\label{prop-union-dimension}
  Let $(P,\le)$ be a poset, and let $\C, \D \subseteq P$ be downsets of dimension $m$ and $n$ respectively. Then $\C \cup \D$ is a downset of dimension at most $m+n$.
\end{proposition}
\begin{proof}
  Certainly $\C \cup \D$ is a downset, so it suffices to show it has dimension at most $m+n$. Let $\{R_1, R_2, \dotsc, R_m\}$ and $\{S_1, S_2, \dotsc, S_n\}$ be realizers of $\C$ and $\D$ respectively. First, note that every member of $\C\setminus \D$ is incomparable with every member of $\D\setminus \C$. Define the refinements
\[
    R_1' = R_1 \oplus (\D\setminus \C),\,\dots,\,
    R_m' = R_m \oplus (\D\setminus \C)
\]
and
\[	
    S_1' = S_1 \oplus (\C\setminus \D),\,\dots,\,
    S_n' = S_n \oplus (\C\setminus \D),
\]
  where $A \oplus B$ is the \emph{ordinal sum} of $A$ and $B$, including all relations within both $A$ and $B$, as well as all relations of the form $a<b$ where $a \in A$ and $b \in B$.
  
  The collection $\{ R_1', \dotsc, R_m', S_1', \dotsc, S_n' \}$ realizes $\C \cup \D$, as it breaks all incomparisons between elements of $\C\setminus \D$ and $\D\setminus \C$ and realizes each of $\C$ and $\D$. This shows that $\C \cup \D$ has dimension at most $m+n$.
\end{proof}

We note that the hypothesis that $\C$ and $\D$ are both downsets in Proposition~\ref{prop-union-dimension} is essential, as shown by the fact that the crown of dimension $n$ can be expressed as the union of two antichains (which are thus each $2$-dimensional).

The downsets of compositions which are not unions of proper subdownsets are precisely the ages, as shown by the following theorem of Fra{\"{\i}}ss\'{e} (which we have specialized to our contexts here). This result implies it suffices to prove Theorems~\ref{thm-part-dim} and \ref{thm-comp-dim} for ages.

\begin{theorem}[Fra{\"{\i}}ss\'{e}~\cite{fraisse:sur-lextension-:}]
\label{thm-atomic-tfae}
The following are equivalent for a downset $\C$ of integer partitions or compositions:
\begin{enumerate}
	\item[(1)] $\C$ cannot be expressed as the union of two proper subdownsets,
	\item[(2)] $\C$ satisfies the \emph{joint embedding property} meaning that for every $a,b\in\C$ there is some $c\in\C$ such that $a,b\le c$, and
	\item[(3)] $\C=\Age(u)$ for some word $u\in\left(\mathbb{P}\cup\{n^\omega\st n\in\mathbb{P}\}\cup\{\omega,\omega^\omega\}\right)^\ast$.
\end{enumerate}
\end{theorem}

We conclude this section by providing the only specific dimension results of the paper. To realize the downset $\Age(\omega\omega)$ of compositions, we use a pair of linear extensions $L_1$ and $L_2$ and a refinement $R_3$. The first, $L_1$, orders compositions according to the \emph{shortlex order}, which sorts compositions first by their length, and within each length sorts compositions according to the lexicographical ordering. The second, $L_2$, orders compositions according to the \emph{shortcolex order}, which sorts compositions first by their length, and within each length sorts compositions according to the colexicographical ordering (lexicographical order, but sorting from right to left). Lastly, the refinement $R_3$ sorts compositions first by their largest part and then by their second largest part. Note that this sometimes leaves a composition and its reverse incomparable, and thus is not a linear extension.

These three refinements constitute a realizer of $\Age(\omega\omega)$, implying that the dimension of $\Age(\omega\omega)$ is at most $3$. Observing that this age contains the crown of dimension $3$ below allows us to conclude that the dimension of $\Age(\omega\omega)$ equals $3$.

  
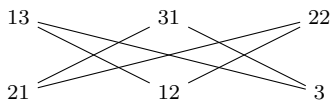
\begin{center}
  \begin{footnotesize}
  \begin{tikzpicture}[xscale=2]
    \hassestp{21,12,3}{13,31,22}
  \end{tikzpicture}
  \end{footnotesize}
  \end{center}

\begin{proposition}
  The dimension of $\Age(\omega\omega)$ is $3$.
\end{proposition}

Similar methods can be applied to show that the dimension of $\Age(\omega 1^\omega)$ is $2$, and that the dimensions of $\Age(1^\omega \omega 1^\omega)$, $\Age(\omega\omega 1^\omega)$, and $\Age(\omega1^\omega\omega)$ are each $4$.

\section{Partitions}
\label{sec-partitions}

Having observed in the introduction that the poset of all integer partitions is infinite dimensional, Theorem~\ref{thm-part-dim} will follow once we show that all proper downsets of partitions are finite dimensional. By Theorem~\ref{thm-atomic-tfae}, every proper downset of partitions can be written as a finite union of ages of the form $\Age(u)$ for some word $u\in \mathbb{P}\cup\{n^\omega\st n\in\mathbb{P}\}\cup\{\omega,\omega^\omega\}$. Because the parts of partitions are ordered, each such age is contained in an age of the form $\Age(\omega^k\lambda\ell^\omega)$ for nonnegative integers $k$ and $\ell$ and a finite partition $\lambda$ whose parts are greater than $\ell$. The Ferrers diagram of the possibly infinite partition $\omega^k \lambda \ell^\omega$ is shown below.

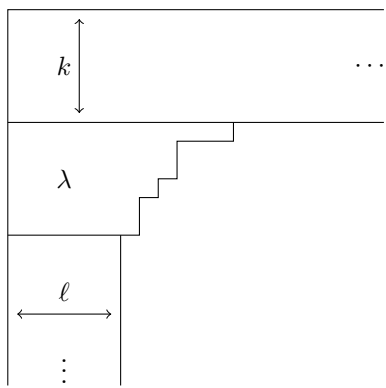
\begin{center}
\begin{tikzpicture}[scale=0.5]
  \draw (0,-10) -- ( 0, 0) -- (10,0);
  \draw (3,-10) -- ( 3,-6);
  \draw (0, -3) -- (10,-3);
  \draw (0, -6) -- (3.5,-6) -- (3.5,-5) -- (4,-5) -- (4,-4.5) -- (4.5,-4.5) -- (4.5,-3.5) -- (6,-3.5) -- (6,-3);
  
  \draw[<->] (1.9,-0.25) -- (1.9,-2.75);
  \draw[<->] (0.25,-8.1) -- (2.75,-8.1);

  \node at (1.5,-1.5) {$k$};
  \node[anchor=east] at (10.4,-1.5) {$\cdots$};
  \node at (1.5,-4.5) {$\lambda$};
  \node at (1.5,-7.5) {$\ell$};
  \node at (1.5,-9.4) {$\vdots$};
  
\end{tikzpicture}
\end{center}

By Proposition~\ref{prop-union-dimension}, it suffices to show that each such age is finite dimensional. We see that $\Age(\omega^k\lambda\ell^\omega)$ is isomorphic (as a poset) to the product $\Age(\omega^k \lambda) \times \Age(\ell^\omega)$. The first of these ages is finite dimensional because it is isomorphic to a subposet of $\mathbb{N}^{k+|\lambda|}$ where $|\lambda|$ denotes the length (number of parts) of $\lambda$. The second of these ages is finite dimensional because it is isomorphic to $\Age(\omega^\ell)$, via conjugation, and that age is in turn isomorphic to a subposet of $\mathbb{N}^{\ell}$. Thus the dimension of $\Age(\omega^k \lambda \ell^\omega)$ is at most $k+\ell+|\lambda|$. This completes the proof of Theorem~\ref{thm-part-dim}.

\section{Compositions}
\label{sec-proof}

We have shown in Section~\ref{sec-intro} that $\Age(\omega\omega\omega)$, $\Age(1^\omega 2 1^\omega 2 1^\omega)$, $\Age(\omega 1^\omega \omega 1^\omega)$, and $\Age(1^\omega\omega 1^\omega\omega)$ are infinite dimensional, and in Section~\ref{sec-tools} we showed that it suffices to show that the maximal ages not containing the four distinguished infinite dimensional ages are finite dimensional. The two types of these maximal ages are those of the forms $\Age(a\omega b 1^\omega c 1^\omega d \omega e)$ and $\Age(a 1^\omega b \omega c \omega d 1^\omega e)$ for finite compositions $a$, $b$, $c$, $d$, and $e$.

We establish the finite dimensionality of these two types of ages with a series of results. Our first such result implies that we may assume $a$ and $e$ are empty.

\begin{proposition}
\label{prop-ku}
  If $\Age(u)$ is finite dimensional for $u\in\left(\mathbb{P}\cup\{1^\omega,\omega\}\right)^\ast$, then $\Age(k u)$ is finite dimensional for all $k\in\mathbb{N}$.
\end{proposition}

\begin{proof}
  We proceed by induction on $k$. The base case of $k=0$ is tautological, so let $k \in \mathbb{P}$ be given, and assume $\Age((k-1)u)$ is finite dimensional. Let $A = \Age(u)$, let $B = \Age(ku) \setminus A$, and for each $1 \le j \le k$, define
  \[
    A_j = \{ j a \in A \}
    \quad\text{and}\quad
    B_j = \{ j a \in B \}.
  \]
  as well as $A_{>k} = \{\ell a \in A \st \ell > k\}$.
  
  By induction, $A \cup B_j$ is finite dimensional for each $1 \le j \le k-1$. Furthermore, $B$ is finite dimensional as it is isomorphic to a subposet of $\mathbb{N} \times A$. Therefore it suffices to show that $A_j \cup B_k$ and $A_{>k} \cup B_k$ are finite dimensional for each $1 \le j \le k$.
  
  Fix $1 \le j \le k$. Given $j a_1 \in A_j$ and $k a_2 \in B_k$, we have $j a_1 \le k a_2$ if and only if $a_1 \le a_2$. For this reason, we define
  \[
    A_j' = \{ a \st j a \in A_j \}
    \quad\text{and}\quad
    B_k' = \{ a \st k a \in B_k \},
  \]
  and consider a realizer $\{L_1, \dots, L_n\}$ of $A_j' \cup B_k'$, which is finite dimensional as it is contained in $A$. For each $1 \le i \le n$, we expand $L_i$ into a linear extension $\hat{L}_i$ of a set containing $A_j \cup B_k$. To do so, we replace the instance of each composition $v$ in $L_i$ with the two element chain $\{jv, kv\}$. If $j a_1 \in A_j$ and $k a_2 \in B_k$ with $j a_1 \not\le k a_2$, then $a_1 \not\le a_2$. Thus $a_2$ precedes $a_1$ in some $L_i$, meaning $k a_2$ precedes $j a_1$ in $\hat{L}_i$.
  
  Lastly, given $\ell a_1 \in A_{>k}$ and $k a_2 \in B_k$, we have $\ell a_1 \le k a_2$ if and only if $\ell a_1 \le a_2$. Let $\{R_1, \dots, R_m\}$ be a realizer of $A_{>k} \cup B_k'$, which is finite dimensional as it is contained in $A$. For each $1 \le i \le m$, we expand $R_i$ into a linear extension $\hat{R}_i$ of $A_{>k} \cup B_k$. To do so, we replace the instance of $a \in B_k'$ in $R_i$ with $k a$.
  
  Then, if $\ell a_1 \in A_{>k}$ and $k a_2 \in B_k$ with $\ell a_1 \not\le k a_2$, then $\ell a_1 \not\le a_2$. Thus $a_2$ precedes $\ell a_1$ in some $R_i$, meaning $k a_2$ precedes $\ell a_1$ in $\hat{R}_i$.
\end{proof}

By applying Proposition~\ref{prop-ku} twice, we obtain the following.

\begin{corollary}
  For all compositions $a$ and $b$, both $\Age(a 1^\omega b)$ and $\Age(a \omega b)$ are finite dimensional.
\end{corollary}

The proof of our next result is more complicated.

\begin{proposition}
\label{prop-1oa1o}
  For all compositions $c$, $\Age(1^\omega c 1^\omega)$ is finite dimensional.  
\end{proposition}
\begin{proof}
  We partition the age of interest into a finite collection of intervals and then construct a family of linear extensions which break the incomparisons between these intervals. These intervals are $[a,1^\omega a 1^\omega) = \{d \in \Age(1^\omega a 1^\omega) \st d \ge a\}$ for each $a = a(1) \cdots a(m) \in \Age(c)$, where the first and last parts of $a$ are at least $2$. Each such interval is itself finite dimensional as it is isomorphic to $\mathbb{N}^2$. Let $\mathcal{R}$ denote the (finite) collection of linear extensions realizing each $[a,1^\omega a 1^\omega)$.
  
  It suffices to consider the union of a pair of such intervals. Let $a,b \le c$ where $a=a(1) \cdots a(m)$ and $b=b(1)\cdots b(n)$ have the property that the first and last parts of each $a$ and $b$ are at least $2$. Note that there are only finitely many such pairs $a,b$ because $c$ is a finite composition. First, if $a$ and $b$ are such that $a \not\le b$, then none of the elements of $[b,1^\omega b 1^\omega)$ embed into any of the elements of $[a,1^\omega a 1^\omega)$, and these incomparisons can be broken with the refinement $[a,1^\omega a 1^\omega)\oplus [b,1^\omega b 1^\omega)$. Let $\mathcal{S}$ be the (finite) collection of these refinements for each $a,b$ with $a \not\le b$.
  
  This leaves us to consider the case where $a$ and $b$ are comparable with  $a<b$, and the only incomparisons left to break are those of the form $1^i a 1^j \not\le 1^k b 1^\ell$.
  
  The bulk of the proof consists of contending with the fact that $a$ may have several embeddings into $b$. Of these, it suffices to consider the \emph{compact} embeddings, meaning those which cannot be shrunk. More precisely, let $\alpha_1<\cdots<\alpha_q$ denote the beginnings of these compact embeddings and $\beta_1<\cdots<\beta_q$ denote the ends. Because these are embeddings, for all $p$ we have
  \begin{align*}
  	a &\le b(\alpha_p) b(\alpha_p+1)\cdots b(\beta_p),\\
    \intertext{and because they are compact, we have both}
  	a &\not\le b(\alpha_p+1) b(\alpha_p+2)\cdots b(\beta_p),\\
  	a &\not\le b(\alpha_p) b(\alpha_p+1)\cdots b(\beta_p-1).
  \end{align*}
  Consider an incomparison between elements of these two intervals, $1^i a 1^j \not\le 1^k b 1^\ell$. This means that, in $\mathbb{N}^2$, we have incomparisons of the form
  \[
  	(i,j)\not\le (k+\alpha_p-1, \ell+n-\beta_p)
  \]
  for each $1\le p\le q$. The set of points $\{(k+\alpha_p-1, \ell+n-\beta_p): 1\le p\le q\}$ is an antichain in $\mathbb{N}^2$ that lies weakly above and to the right of $(k,\ell)$ in the plane, as shown on the left of Figure~\ref{fig-tile-plane}.
  
\NewDocumentCommand{\tilingshift}{ s m m }{

  \absdothollow{(#2,#3)}
  \IfBooleanT{#1}{\node[anchor=north] at (#2,#3) {$(k,\ell)$};}
  
  \draw [white, line cap = round, fill = lightgray] 
    (2+#2,6+#3) -- (2+#2,5+#3) -- (4+#2,5+#3) -- (4+#2,3+#3) -- (6+#2,3+#3) -- (6+#2,2+#3) -- (7+#2,2+#3) -- (7+#2,6+#3) -- (2+#2,6+#3);
  
  \draw [dotted, thick, line cap = round] 
    (2+#2,6+#3) -- (2+#2,5+#3) -- (4+#2,5+#3) -- (4+#2,3+#3) -- (6+#2,3+#3) -- (6+#2,2+#3) -- (7+#2,2+#3);
  \draw [thick, line cap = round] (2+#2,6+#3) -- (7+#2,6+#3) -- (7+#2,2+#3);

  \absdot{(2+#2,6+#3)}
  \absdot{(4+#2,5+#3)}
  \absdot{(6+#2,3+#3)} 
  \absdot{(7+#2,2+#3)}
  
  \IfBooleanT{#1}{\node [above left] at (6.6+#2,4+#3) {$T_{k,\ell}$};}
  
}

\begin{figure}
  \begin{footnotesize}
    \begin{center}
      \begin{tikzpicture}[scale=0.34]
        
        \draw[thick, <->] (0,11) -- (0,0) -- (12,0); 
  
     	\absdothollow{(3,2)}
        \node[anchor=north] at (3,2) {$(k,\ell)$};
        
        \absdot{(5,8)}
        \absdot{(7,7)}
        \absdot{(9,5)} 
        \absdot{(10,4)}
        
        \node[anchor=south] at (5,8) {$(k+\alpha_1-1,\ell+n-\beta_1)$};
        \node[anchor=north] at (10,4) {$(k+\alpha_4-1,\ell+n-\beta_4)$};
  
      \end{tikzpicture}
    \quad\quad
      \begin{tikzpicture}[scale=0.34]
        \draw[thick, <->] (0,11) -- (0,0) -- (12,0);
        \tilingshift*{3}{2}
      \end{tikzpicture}
    \quad\quad
      \begin{tikzpicture}[scale=0.17]
        \draw[thick, <->] (0,22) -- (0,0) -- (24,0);
        
        \draw[dotted] ( 2,0) -- ( 2,22);
        \draw[dotted] (11,0) -- (11,22);
        \draw[dotted] (20,0) -- (20,22);
        
        \draw[dotted] (0, 1) -- (24, 1);
        \draw[dotted] (0, 9) -- (24, 9);
        \draw[dotted] (0,17) -- (24,17);
        
        \tilingshift{3}{2}
        \tilingshift{3}{10}
        \tilingshift{12}{2}
        \tilingshift{12}{10}
      \end{tikzpicture}  
    \end{center}
  \end{footnotesize}
  
  \caption{(Left) A point $(k,\ell)$ representing $1^k b 1^\ell$ together with associated points representing the minimal compositions of the form $1^i a 1^j$ which do not embed into $1^k b 1^\ell$. (Center) A point $(k,\ell)$ representing $1^k b 1^\ell$ and its associated set $T_{k,\ell}$ representing compositions of the form $1^i a 1^j$. (Right) The shaded regions indicate part of a family of compositions included in one refinement constructed at the end of the proof of Proposition~\ref{prop-1oa1o}.}
  \label{fig-tile-plane}
\end{figure}
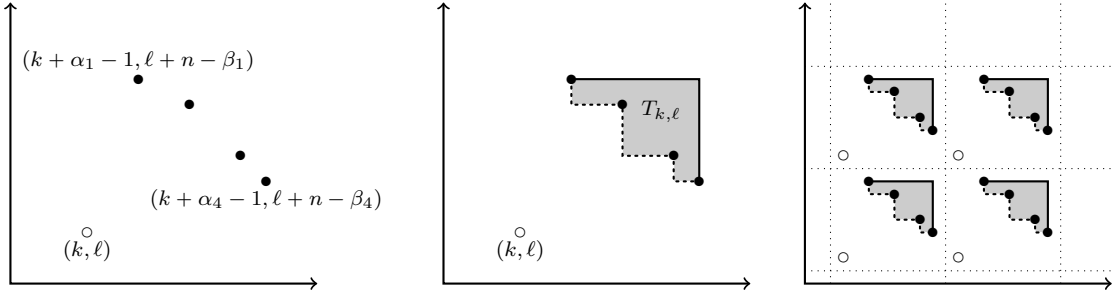
  
  We now introduce two refinements of $[a,1^\omega a 1^\omega)\cup [b,1^\omega b 1^\omega)$. The first sorts compositions by the largest $r$ such that $1^ra$ is contained in them, while the second sorts compositions by the largest $s$ such that $a1^s$ is contained in them. For a given $k$ and $\ell$, these two refinements break all incomparisons of the form $1^i a 1^j\not\le 1^k b 1^\ell$ where $i > k+\alpha_q-1$ or $j>\ell+n-\beta_1$. Still thinking of $k$ and $\ell$ as fixed, this leaves us with a finite set of incomparisons of the form $1^i a 1^j \not\le 1^k b 1^\ell$ to break, as illustrated in the center of Figure~\ref{fig-tile-plane}. Let $T_{k,\ell}$ denote the finite set of compositions of the form $1^i a 1^j$ whose incomparisons with $1^k b 1^\ell$ have not been dealt with. Thus $T_{k,\ell}$ is the set
  \[
  	\{ 1^i a 1^j \st \text{$(i,j) \le (k + \alpha_q - 1, \ell + n - \beta_1)$ and $(i,j) \not\le (k+\alpha_p-1,\ell+n-\beta_p)$ for all $1 \le p \le q$}\}.
  \]
  We identify each composition $1^i a 1^j\in T_{k,\ell}$ with the point $(i,j)$ in the plane. Thus the points corresponding to the compositions in $T_{k,\ell}$ are contained in the rectangle
	\[
      [k, k+\alpha_q-1] \times [\ell, \ell+n-\beta_1].
	\]
  
  Given $k,\ell$, we define a refinement $R_{k,\ell}$ of $\{1^k b 1^\ell\} \cup T_{k,\ell}$ in which $1^k b 1^\ell$ is less than each element of $T_{k,\ell}$. All that remains is to combine the collection of refinements $R_{k,\ell}$ into finitely many refinements of $[a,1^\omega a 1^\omega)\cup [b,1^\omega b 1^\omega)$. We achieve this by partitioning $\mathbb{N}^2$ into equivalence classes with respect to the equivalence relation $(k,\ell) \sim (k',\ell')$ if $k \equiv k' \operatorname{mod} \alpha_q$ and $\ell \equiv \ell' \operatorname{mod} n-\beta_1+1$. We further write $[(k,\ell)]$ to denote the equivalence class containing $(k,\ell)$. Note that there are only finitely many such equivalences classes.
  
  The motivation for this equivalence relation is that if $(k,\ell)\sim (k',\ell')$ then the relations defined by $R_{k,\ell}$ and $R_{k',\ell'}$ do not conflict. Thus for any $(k,\ell) \in \mathbb{N}^2$, all of the relations
  \[
    \bigcup_{(k',\ell') \in [(k,\ell)]} R_{k',\ell'}
  \]
  can be combined into a single refinement. The compositions involved in one such refinement are drawn on the right of Figure~\ref{fig-tile-plane}.
  
  As there are only finitely many such equivalence classes in $\mathbb{N}^2$, and only finitely many pairs $a,b$ with $a \le b \le c$, this (finite) set of refinements, together with the refinements of $\mathcal{R}$ and $\mathcal{S}$, realizes $\Age(1^\omega c 1^\omega)$, completing the proof.
\end{proof}

With Proposition~\ref{prop-1oa1o} established, showing that ages of the forms $\Age(\omega a 1^\omega b 1^\omega c \omega)$ and $\Age(1^\omega a \omega b \omega c 1^\omega)$ are finite dimensional is accomplished by first proving that ages of the forms $\Age(\omega a 1^\omega b 1^\omega)$ and $\Age(1^\omega a \omega b 1^\omega)$ are finite dimensional. Each of these steps relies on Proposition~\ref{prop-1oa1o}.

\begin{proposition}
\label{prop-1oa1obo}
  For all compositions $a$ and $b$, $\Age(\omega a 1^\omega b 1^\omega)$ is finite dimensional.  
\end{proposition}
\begin{proof}
  Let $m$ denote the maximum entry in $a$ or $b$ and let $\mbar = m + 1$. By Propositions~\ref{prop-ku} and \ref{prop-1oa1o} we have that $\Age(\mbar a 1^\omega b 1^\omega)$ is finite dimensional, so let $\{L_1, \dots L_n\}$ be a realizer of it. For each $1 \le i \le n$, we expand $L_i$ into a linear extension $\hat{L}_i$ of $\Age(\omega a 1^\omega b 1^\omega)$. To do so, we replace the instance of $\mbar x$ in $L_i$ with the linearly ordered interval $[\mbar x, \omega x)$.
  
  The only incomparisons yet to be handled are those of the form $k_1 x_1 \not\le k_2 x_2$ where $\mbar x_1 \le \mbar x_2$ and $k_1 > k_2 \ge \mbar$. These are fixed by including a single refinement which sorts elements of $\Age(\omega a 1^\omega b 1^\omega)$ by their largest entry.
\end{proof}

\begin{proposition}
\label{prop-oa1ob1oco}
  For all compositions $a$, $b$, and $c$, $\Age(\omega a 1^\omega b 1^\omega c \omega)$ is finite dimensional.
\end{proposition}

\begin{proof}
  We proceed by defining six sets, each of which is finite dimensional and whose union is the age of interest, and then construct a family of refinements which break the incomparisons between the sets. Let $m$ denote the maximum entry in $a$, $b$, or $c$, let $\mbar = m+1$, let $\mbarbar = \mbar+1$, and define
  \[
    \begin{array}{rlcl}
      &A   &=& [\varepsilon, m a 1^\omega b 1^\omega c m),\\
      &B_1 &=& [\mbar, m a 1^\omega b 1^\omega c \mbar),\\
      &B_2 &=& [m, m a 1^\omega b 1^\omega c m),\\
      &C_1 &=& [\mbar \mbar, \omega a 1^\omega b 1^\omega c \mbar),\\     
      &C_2 &=& [\mbar \mbar, \mbar a 1^\omega b 1^\omega c \omega),\\
      &D   &=& [\mbarbar \mbarbar, \omega a 1^\omega b 1^\omega c \omega).
    \end{array}
  \]
  Now, the complement of $D$,
  \begin{align*}
    \Age(\omega a 1^\omega b 1^\omega c \omega) \setminus D
      &= A \cup B_1 \cup B_2 \cup C_1 \cup C_2 \\
      &= \Age(\omega a 1^\omega b 1^\omega c \mbar) \cup \Age(\mbar a 1^\omega b 1^\omega c \omega)
  \end{align*}
  is finite dimensional by Propositions~\ref{prop-union-dimension}, \ref{prop-ku}, and \ref{prop-1oa1obo}. Also, $C_1 \cup C_2 \cup D$ is finite dimensional as it is isomorphic to a subposet of $\mathbb{N} \times \Age(a 1^\omega b1^\omega c) \times \mathbb{N}$. Thus it suffices to show that the incomparisons between $A\cup B_1 \cup B_2$ and $D$ can be broken with finitely many refinements. 
    
  Let $\{L_1, \dots, L_n\}$ be a realizer for $\Age(\mbar a 1^\omega b 1^\omega c \mbar)$. For each $1 \le i \le n$, we expand $L_i$ into a refinement $\hat{L}_i$ of $\Age(\omega a 1^\omega b 1^\omega c \omega)$. To do so, for each $v$, we replace the instance of $\mbar v \mbar$ in $L_i$ with the interval $[\mbar v \mbar, \omega v \omega)$. If $u \in A \cup B_1 \cup B_2$ and $k_1 v k_2 \in D$ for integers $k_1, k_2 \ge \mbarbar$ and $u \not\le k_1 v k_2$, then $u \not\le \mbar v\mbar$, so $\mbar v \mbar$ is less than $u$ in some $L_i$, and thus $k_1 v k_2$ is less than $u$ in $\hat{L}_i$. This completes the proof.
\end{proof}

The proofs of our next two results are very similar to those of Propositions~\ref{prop-1oa1obo} and \ref{prop-oa1ob1oco}.

\begin{proposition}
\label{prop-1oaob1o}
  For all compositions $a$ and $b$, $\Age(1^\omega a \omega b 1^\omega)$ is finite dimensional.
\end{proposition}
\begin{proof}
  Let $m$ denote the maximum entry in $a$ or $b$ and let $\mbar = m + 1$. By Proposition~\ref{prop-1oa1o} we have that $\Age(1^\omega a \mbar b 1^\omega)$ is finite dimensional, so let $\{L_1, \dots L_n\}$ be a realizer of it. For each $1 \le i \le n$, we expand $L_i$ into a linear extension $\hat{L}_i$ of $\Age(1^\omega a \omega b 1^\omega)$. To do so, we replace the instance of $x \mbar y$ in $L_i$ with the linearly ordered interval $[x \mbar y, x \omega y)$. 
  
  The only incomparisons yet to be handled are those of the form $x_1 k_1 y_1 \not\le x_2 k_2 y_2$ where $x_1 \mbar y_1 \le x_2 \mbar y_2$ and $k_1 > k_2 \ge \mbar$. These are fixed by including a single refinement which sorts elements of $\Age(1^\omega a \omega b 1^\omega)$ by their largest entry.
\end{proof}

\begin{proposition}
\label{prop-1oaoboc1o}
  For all compositions $a$, $b$, and $c$, $\Age(1^\omega a \omega b \omega c 1^\omega)$ is finite dimensional.  
\end{proposition}
\begin{proof}
Let $m$ denote the maximum entry in $a$, $b$, or $c$, let $\mbar = m+1$, let $\mbarbar = \mbar+1$, and define
  \[
    \begin{array}{rlcl}
      &A   &=& [\varepsilon, 1^\omega a m b m c 1^\omega),\\
      &B_1 &=& [\mbar, 1^\omega a m b \mbar c 1^\omega),\\
      &B_2 &=& [\mbar, 1^\omega a \mbar b m c 1^\omega),\\
      &C_1 &=& [\mbar \mbar, 1^\omega a \omega b \mbar c 1^\omega),\\
      &C_2 &=& [\mbar \mbar, 1^\omega a \mbar b \omega c 1^\omega),\\
      &D   &=& [\mbarbar \mbarbar, 1^\omega a \omega b \omega c 1^\omega).
    \end{array}
  \]
  The complement of $D$,
  \begin{align*}
    \Age(1^\omega a \omega b \omega c 1^\omega)
      &= A \cup B_1 \cup B_2 \cup C_1 \cup C_2 \\
      &= \Age(1^\omega a \mbar b \omega c 1^\omega) \cup \Age(1^\omega a \omega b \mbar c 1^\omega)
  \end{align*}
  is finite dimensional by Propositions~\ref{prop-union-dimension}, \ref{prop-ku}, and \ref{prop-1oaob1o}. Also, $C_1 \cup C_2 \cup D$ is finite dimensional as it is isomorphic to a subposet of $\Age(1^\omega a) \times \mathbb{N} \times \Age(b) \times \mathbb{N} \times \Age(c 1^\omega)$. Thus it suffices to show that the incomparisons between $A\cup B_1 \cup B_2$ and $D$ can be broken with finitely many refinements. 
    
  Let $\{L_1, \dots, L_n\}$ be a realizer for $\Age(1^\omega a \mbar b \mbar c 1^\omega)$. For each $1 \le i \le n$, we expand $L_i$ into a refinement $\hat{L}_i$ of $A \cup B_1 \cup B_2 \cup D$. To do so, for each $x,y,z$, we replace the instance of $x \mbar y \mbar z$ in $L_i$ with the interval $[x \mbar y \mbar z, x \omega y \omega z)$. 
  
  Then, if $u \in A \cup B_1 \cup B_2$ and $x k_1 y k_2 z \in D$ with $k_1, k_2 \ge \mbarbar$, and $u \not\le x k_1 y k_2 z$, then we have $u \not\le x \mbar y \mbar z$. Thus $x \mbar y \mbar z$ is less than $u$ in some $L_i$, and thus $x k_1 y k_2 z$ is less than $u$ in $\hat{L}_i$. This completes the proof.
\end{proof}

With Propositions~\ref{prop-oa1ob1oco} and \ref{prop-1oaoboc1o} established, we note that the proof of Theorem~\ref{thm-comp-dim} is complete, given the remarks at the beginning of Section~\ref{sec-proof}.

\section{Concluding Remarks}
\label{sec-conclusion}

Theorems~\ref{thm-part-dim} and \ref{thm-comp-dim} characterize the finite dimensional downsets in the posets of integer partitions and compositions, respectively. There are several similar contexts in which the analogous questions have yet to be considered. One such context is the poset of permutations under the permutation pattern order. We refer to the second author's survey~\cite{vatter:permutation-cla:} for more information on this order. A related example is the poset of set partitions, first studied by Klazar~\cite{klazar:counting-patter:2,klazar:counting-patter:1,klazar:on-abab-free-an:} and Sagan~\cite{sagan:pattern-avoidan:}. Another natural context would be the generalized subword order over an arbitrary poset $P$, a context where McNamara and Sagan~\cite{mcnamara:the-mobius-func:} have recently determined the M\"obius function. Indeed, even the special case of words over a two-element antichain appears to be untouched.

\bibliographystyle{acm}
\bibliography{refs}

\end{document}